\documentclass[11pt]{amsart}

\usepackage[T1]{fontenc}
\usepackage{lmodern}
\usepackage{microtype}
\usepackage{amsmath,amssymb,amsthm,mathtools,mathrsfs}
\usepackage[letterpaper,left=0.90in,right=0.90in,top=1in,bottom=1in]{geometry}
\usepackage[colorlinks=true,allcolors=blue]{hyperref}
\usepackage[nameinlink,capitalise,noabbrev]{cleveref}
\usepackage{tikz-cd}

\makeatletter
\renewcommand{\subsection}{%
  \@startsection{subsection}{2}{\z@}%
  {.5\linespacing\@plus.7\linespacing}%
  {.3\linespacing}%
  {\normalfont\bfseries}%
}
\makeatother

\newtheorem{theorem}{Theorem}[section]
\newtheorem{proposition}[theorem]{Proposition}
\newtheorem{corollary}[theorem]{Corollary}
\newtheorem{lemma}[theorem]{Lemma}
\theoremstyle{remark}
\newtheorem{remark}[theorem]{Remark}

\newcommand{\Q}{\mathbb Q}
\newcommand{\C}{\mathbb C}
\newcommand{\A}{\mathbb A}

\newcommand{\Perv}{\operatorname{Perv}}

\newcommand{\Sh}{\operatorname{Sh}}
\newcommand{\chiorb}{\chi^{\mathrm{orb}}}

\newcommand{\Rm}[1]{\mathrm{#1}}

\newcommand{\Bf}[1]{\mathbf{#1}}
\newcommand{\sh}[1]{\mathcal{#1}}
\newcommand{\bb}[1]{\mathbb{#1}}
\newcommand{\fr}[1]{\mathfrak{#1}}
\newcommand{\tab}{\hspace*{\parindent}}

\title{Perverse Euler Characteristics of Hermitian Locally Symmetric Spaces}

\author{Mert Akdenizli}
\address{Department of Mathematics, Purdue University, West Lafayette, Indiana, USA}
\email{makdeniz@purdue.edu}

\author{Quoc-Anh Tran}
\address{Department of Mathematics,
University of Chicago, Chicago, Illinois, USA}
\email{quocanh@uchicago.edu}

\begin{document}

\begin{abstract}
    We prove that finite-volume locally Hermitian symmetric spaces of noncompact type have nonnegative perverse Euler characteristics. To show this, we obtain a nefness result for the logarithmic cotangent bundle of a smooth toroidal compactification. Combining this with a positivity criterion for Euler characteristics of perverse sheaves, we deduce the nonnegativity result. We further prove that the inequality is strict for perverse sheaves with full support. As applications, we get nonnegativity results for perverse Euler characteristics on various moduli spaces.
\end{abstract}

\maketitle
\tableofcontents

\section{Introduction and main results}
\label{sec:introduction}

Let $\mathcal D$ be a Hermitian symmetric domain of noncompact type and let
$$
  G_{\mathcal D}=\operatorname{Aut}(\mathcal D)^{\circ}.
$$

For a lattice $\Gamma\subset G_{\mathcal D}$, the quotient
$\Gamma\backslash\mathcal D$ is a finite-volume locally Hermitian symmetric orbifold.  We study the nonnegativity of Euler characteristics of perverse
sheaves on the Deligne--Mumford stack attached to this orbifold. Let
$$
  \mathcal X_\Gamma
  :=
  [\mathcal D/\Gamma]
$$

denote the associated analytic Deligne--Mumford quotient stack.  Its
coarse space is the locally symmetric quotient $\Gamma\backslash\mathcal D.$ All constructible complexes and perverse sheaves on
\(\mathcal X_\Gamma\) are understood in the analytic topology.  We write
$$
  \chiorb(\mathcal X_\Gamma,K)
$$

for the Euler--Satake characteristic of a constructible complex
\(K\in D^b_c(\mathcal X_\Gamma,\Q)\).  Its definition and its behavior
under finite \'etale covers are recalled in
Section~\ref{subsec:PEC}. If \(\Gamma\) is torsion-free, then \(\mathcal X_\Gamma\) is represented
by the complex manifold $X_\Gamma=\Gamma\backslash\mathcal D,$ which carries its natural smooth quasi-projective algebraic structure.
In this case \(\chiorb\) is the ordinary Euler characteristic.

Our main result is the following:

\begin{theorem}[Main theorem]
\label{thm:main}
Let $\mathcal D$ be a Hermitian symmetric domain of noncompact type and let
$\Gamma\subset\operatorname{Aut}(\mathcal D)^{\circ}$ be a lattice.  Then
$$
  \chiorb(\mathcal X_\Gamma,P)\geq 0
$$
for every $P\in\Perv(\mathcal X_\Gamma,\Q).$ If $\Gamma$ is torsion-free, so that $X_\Gamma=\Gamma\backslash\mathcal D$ is a smooth quasi-projective variety, then
$$
  \chi(X_\Gamma,P)\geq 0
  \qquad
  \text{for every }P\in\Perv(X_\Gamma,\Q).
$$

\end{theorem}

If \(X\) is a smooth irreducible complex variety of dimension
\(n\) and \(P\in\Perv(X,\Q)\) has full support, then there is a dense
Zariski-open subset \(U\subset X\) such that $P|_U\simeq L[n]$ for a local system \(L\) on \(U\). The rank of \(L\) is called the
generic rank of \(P\) and is denoted by \(r(P)\).  If
\(\operatorname{Supp}(P)\neq X\), we set \(r(P)=0\).

For a perverse sheaf on \(\mathcal X_\Gamma\), we define its generic rank to be the generic rank of its pullback to a neat finite \'etale cover \(Y\to\mathcal X_\Gamma\).  This is independent of the chosen cover because any two finite tower admits a common one finite \,etale pullback does not change the rank of a local system.

Following \cite{ArapuraMesePatel} we also improve the result to strict inequality for perverse sheaves with full support. 

\begin{proposition}
\label{prop:strict}
    With the notation of Theorem~\ref{thm:main}, let \(n=\dim_{\C}\mathcal D\). Then for every \(P\in\Perv(\mathcal X_\Gamma,\Q)\),
    \[\chiorb(\mathcal X_\Gamma,P) \geq r(P)(-1)^n\chiorb(\mathcal X_\Gamma).\]
    In particular for a perverse sheaf $P$ which has full support we have
    \[
\chiorb(\mathcal X_\Gamma,P)>0.\]
\end{proposition}

The stack formulation is needed when the lattice has torsion.  In
particular, the first assertion of Theorem~\ref{thm:main} concerns the
Deligne-Mumford quotient $\mathcal X_\Gamma$, rather than the possibly
singular coarse quotient $\Gamma\backslash\mathcal D$.

Theorem~\ref{thm:main} gives a finite-level statement for Shimura
varieties.  Let $(G,\mathcal X)$ be a Shimura datum and let
$K\subset G(\mathbb A_f)$ be a compact open subgroup.  Choose a normal neat compact open subgroup $K'\triangleleft K$ and define the finite-level
Deligne-Mumford Shimura stack by
\begin{equation}
\Sh_K^{\mathrm{DM}}(G,\mathcal X)
:=
\bigl[\Sh_{K'}(G,\mathcal X)/(K/K')\bigr].
\label{eq:DM-Shimura-stack}
\end{equation}
This is an algebraic Deligne-Mumford stack.  Its coarse moduli space is the usual finite-level Shimura variety
$\Sh_K(G,\mathcal X)$.

\begin{corollary}
\label{cor:shimura}
Let $(G,\mathcal X)$ be a Shimura datum and let
$K\subset G(\mathbb A_f)$ be an arbitrary compact open subgroup.  Then
$$
  \chiorb\bigl(\Sh_K^{\mathrm{DM}}(G,\mathcal X)_{\C},P\bigr)
  \geq0
$$
for every
$$P\in\Perv\bigl(\Sh_K^{\mathrm{DM}}(G,\mathcal X)_{\C},\Q\bigr).$$

Strict inequality holds whenever $P$ has full support.
When $K$ is neat, the stack
$\Sh_K^{\mathrm{DM}}(G,\mathcal X)_{\C}$ is represented by the usual smooth
finite-level Shimura variety, and $\chiorb$ becomes the ordinary Euler
characteristic.
\end{corollary}

The geometric input needed is the following nefness theorem for
toroidal compactifications of Shimura varieties.

\begin{theorem}
\label{thm:shimura-log-nef}
Let $X$ be a connected component of a Shimura variety at neat finite level,
and let
$$
  j\colon X\hookrightarrow\overline X
$$
be a smooth projective toroidal compactification whose boundary $D=\overline X\setminus X$ is a simple normal-crossings divisor.  Then
$$
  \Omega^1_{\overline X}(\log D)
$$
is nef.
\end{theorem}

For complex-ball quotients, the corresponding statement is due to Cadorel.

\begin{theorem}[Cadorel]
\label{thm:cadorel-ball}
Let $\Delta\subset\operatorname{Aut}(\mathbb B^n)$ be a finite-volume lattice
whose parabolic elements are unipotent, and let
$$
  X=\Delta\backslash\mathbb B^n.
$$
Let
$$
  X\hookrightarrow\overline X
$$
be its smooth projective toroidal compactification and let $D=\overline X\setminus X$ be the boundary.  Then
$$
  \Omega^1_{\overline X}(\log D)
  \quad\text{is nef}.
$$

\end{theorem}

\begin{proof}
This is \cite[Theorem~3]{Cadorel}.
\end{proof}

Section~\ref{sec:preliminaries} will reduce the main problem to finite level Shimura variety and a ball quotient case. After obtaining the nefness results in each case the product nefness will follow. The logarithmic characteristic-cycle criterion of Arapura--Patel \cite{ArapuraPatel} and finite \'etale
invariance will give the desired result on $\mathcal X_\Gamma$. 

Nefness results on finite level Shimura setting is proved using the adjoint variation of Hodge structure and positivity results on its canonical extension.  Together with Theorem~\ref{thm:cadorel-ball}, it supplies the two geometric inputs in
the preceding reduction.  The final part of the paper applies the resulting Euler-characteristic inequalities to moduli spaces admitting locally symmetric period descriptions.

\section{Preliminaries and the reduction theorem}
\label{sec:preliminaries}

\subsection{The PEC property and the logarithmic criterion}
\label{subsec:PEC}

Let $U$ be a complex algebraic variety and let
$K\in D^b_c(U,\Q)$. We have
\[
  \chi(U,K)
  =
  \sum_{j\in\mathbb Z}(-1)^j
  \dim_{\Q}\mathbb H^j(U,K).
\]

We say that $U$ has the perverse Euler-characteristic property, or
briefly the PEC property, if
\begin{equation}
  \chi(U,P)\geq0
  \qquad
  \text{for every }P\in\Perv(U,\Q).
  \label{eq:PEC-variety}
\end{equation}

The following logarithmic criterion will be used throughout the paper.

\begin{theorem}[Arapura--Patel lemma 2.2 \cite{ArapuraPatel}]
\label{thm:AP}
Let
\[
  U=\overline U\setminus D,
\]
where $\overline U$ is a smooth projective complex variety and $D$ is a
simple normal-crossings divisor.  If
\[
  \Omega^1_{\overline U}(\log D)
\]
is nef, then $U$ has the PEC property; equivalently,
\[
  \chi(U,P)\geq0
  \qquad
  \text{for every }P\in\Perv(U,\Q).
\]
\end{theorem}

Let $\mathcal Y$ be a finite-type complex Deligne--Mumford stack. For a complex Deligne--Mumford stack, all constructible complexes and
perverse sheaves below are taken on the underlying analytic stack with
\(\Q\)-coefficients. Following
\cite[\S2.6]{ArapuraPatel}, we write
\[
  \chiorb(\mathcal Y,K)\in\Q
\]
for the Euler--Satake characteristic of
$K\in D^b_c(\mathcal Y,\Q)$.  We say that $\mathcal Y$ has the PEC property if
\begin{equation}
  \chiorb(\mathcal Y,P)\geq0
  \qquad
  \text{for every }P\in\Perv(\mathcal Y,\Q).
  \label{eq:PEC-stack}
\end{equation}
When $\mathcal Y$ is a variety, its Euler--Satake characteristic is the
ordinary Euler characteristic
\cite[Theorem~2.5\textup{(i)}]{ArapuraPatel}.  More generally, if
\[
  \mathcal Y=[Y/F]
\]
for a finite group $F$ acting on a variety $Y$, and
$q\colon Y\to\mathcal Y$ is the quotient atlas, then
\begin{equation}
  \chiorb(\mathcal Y,K)
  =
  \frac{1}{|F|}\chi(Y,q^*K).
  \label{eq:quotient-stack-Euler-Satake}
\end{equation}
This is the finite \'etale degree formula of
\cite[Theorem~2.5\textup{(iii)}]{ArapuraPatel}.

We use the perverse $t$-structure on Deligne--Mumford stacks in its
smooth-local form; see \cite[\S4]{LaszloOlsson}.  Thus, if
$u\colon U\to\mathcal Y$ is a smooth atlas of pure relative dimension $e$,
then $P$ is perverse on $\mathcal Y$ if and only if $u^*P[e]$ is perverse on
$U$.  In particular, pullback by a representable \'etale morphism is
$t$-exact.

All locally symmetric Deligne--Mumford stacks considered below admit
representable finite \'etale atlases by varieties, by the quotient
presentation above.  The substacks occurring in the applications inherit such atlases by base change. We record finite
\'etale invariance in this setting.

\begin{proposition}[Finite \'etale invariance of PEC]
\label{prop:PEC-finite-etale}
Let
\[
  f\colon\mathcal Y\longrightarrow\mathcal X
\]
be a representable finite \'etale surjective morphism of finite-type complex
Deligne--Mumford stacks of constant degree $d$.  Assume that $\mathcal X$
admits a representable finite \'etale atlas by a variety.  Then
\[
  \mathcal X\text{ has PEC}
  \quad\Longleftrightarrow\quad
  \mathcal Y\text{ has PEC}.
\]
\end{proposition}

\begin{proof}
Suppose first that $\mathcal Y$ has PEC, and let
$P\in\Perv(\mathcal X,\Q)$.  Since $f$ is \'etale, $f^*P$ is perverse.  The
finite \'etale degree formula
\cite[Theorem~2.5\textup{(iii)}]{ArapuraPatel} gives
\[
  \chiorb(\mathcal Y,f^*P)
  =
  d\,\chiorb(\mathcal X,P).
\]
The left-hand side is nonnegative, so
$\chiorb(\mathcal X,P)\geq0$.

Conversely, suppose that $\mathcal X$ has PEC, and let
$Q\in\Perv(\mathcal Y,\Q)$. We can choose a finite \'etale surjective atlas $u:U\to \mathcal X$ of degree $m$ with $U$ a variety. Consider the cartesian square
\[
\begin{tikzcd}
V:=U\times_{\mathcal X}\mathcal Y
  \arrow[r,"v"]
  \arrow[d,"f'"']
&
\mathcal Y
  \arrow[d,"f"]
\\
U
  \arrow[r,"u"']
&
\mathcal X .
\end{tikzcd}
\]
Since \(f\) is representable, \(V\) is an algebraic space, and since
\(f'\colon V\to U\) is finite, \(V\) is in fact a variety.  Moreover,
\(f'\) is finite \'etale of degree \(d\), while \(v\) is finite \'etale
of degree \(m\), being the respective base changes of \(f\) and \(u\). By the finite \'etale degree
formula, proper base change, and the usual equality of Euler characteristics
under finite pushforward, one has
\begin{align*}
  m\,\chiorb(\mathcal Y,Q)
  &=\chi(V,v^*Q)\\
  &=\chi(U,Rf'_*v^*Q)\\
  &=\chi(U,u^*Rf_*Q)\\
  &=m\,\chiorb(\mathcal X,Rf_*Q).
\end{align*}
Now $Rf_*Q\in \Perv(\mathcal X,\Q)$. To see this, we need to show that $u^*Rf_*Q$ is perverse on $U$ a variety. Now $f$ is finite, and thus proper. We can use proper base change which says $u^*Rf_*Q\simeq Rf'_*v^*Q$. Since $v$ is \'etale $v^*Q$ is perverse on $V$ which is a variety and as $f'$ is finite we get that $Rf'_*v^*Q\in \Perv(U,\Q)$.
Therefore
\[
  \chiorb(\mathcal Y,Q)
  =
  \chiorb(\mathcal X,Rf_*Q)
  \geq0,
\]
and $\mathcal Y$ has PEC.
\end{proof}

\subsection{The common-cover reduction theorem}
\label{subsec:common-cover}

We reduce to a finite \'etale cover which is a product of neat locally
symmetric varieties.

\begin{proposition}
\label{prop:neat-product-cover}
Let \(\mathcal D\) be a Hermitian symmetric domain of noncompact type and let $\Gamma\subset G_{\mathcal D}=\operatorname{Aut}(\mathcal D)^\circ$ be a lattice.  There exist decompositions
$$
  \mathcal D=\mathcal D_1\times\cdots\times\mathcal D_r,
  \qquad
  \Lambda=\Lambda_1\times\cdots\times\Lambda_r\subset\Gamma,
$$
where \(\Lambda\) has finite index in \(\Gamma\) and each \(\Lambda_i\)
is neat, such that, with
$$
  Y_i=\Lambda_i\backslash\mathcal D_i,
  \qquad
  Y=Y_1\times\cdots\times Y_r,
$$
the variety \(Y\) is smooth and quasi-projective and the natural map
$$
  Y\longrightarrow[\mathcal D/\Gamma]
$$

is representable, finite \'etale, and surjective. Each \(Y_i\) is either an arithmetic locally symmetric variety or a finite-volume complex-ball quotient with unipotent parabolic elements.
\end{proposition}

\begin{proof}
The group \(G_{\mathcal D}\) is adjoint semisimple and has no compact
factors.  By \cite[Proposition~4.3.3]{WitteMorris}, there is a
decomposition $G_{\mathcal D}=G_1\times\cdots\times G_r$ such that
$$
  \Gamma'
  =\Gamma_1\times\cdots\times\Gamma_r
  \subset\Gamma,
  \qquad
  \Gamma_i=\Gamma\cap G_i,
$$

has finite index and each \(\Gamma_i\) is an irreducible lattice.  The
corresponding decomposition of the symmetric space is
\(\mathcal D=\prod_i\mathcal D_i\).

If \(\operatorname{rank}_{\mathbb R}G_i\geq2\), then \(\Gamma_i\) is
arithmetic by Margulis arithmeticity
\cite[Theorem~1]{MargulisArithmeticity}.  Thus a nonarithmetic
\(\Gamma_i\) has real rank one; since \(\mathcal D_i\) is Hermitian,
the classification of irreducible Hermitian symmetric spaces gives $\mathcal D_i\simeq\mathbb B^{n_i};$ see \cite[Chapters~VIII and~X]{Helgason}.

For an arithmetic factor choose a neat finite-index subgroup
\(\Lambda_i\subset\Gamma_i\) using
\cite[Proposition~3.5]{MilneShimura}; its quotient is smooth and
quasi-projective by the Baily--Borel theorem
\cite{BailyBorel}.  For a nonarithmetic ball factor, choose a finite index neat subgroup \(\Lambda_i\subset\Gamma_i\), such a subgroup exists by \cite[Remark~A.2]{Deng}. Neatness implies that \(\Lambda_i\) is torsion-free and that all of its parabolic elements are unipotent; see \cite[\S A.1]{Deng}. The quotient \(\Lambda_i\backslash\mathbb B^{n_i}\) is smooth and quasi-projective by \cite{Mok}.

Hence $\Lambda=\Lambda_1\times\cdots\times\Lambda_r $
is finite index and
$$
  Y=\Lambda\backslash\mathcal D
   =\prod_{i=1}^r\Lambda_i\backslash\mathcal D_i
$$
is a smooth quasi-projective variety.  Finally, the inclusion
\(\Lambda\subset\Gamma\) induces
$$
  [\mathcal D/\Lambda]\longrightarrow[\mathcal D/\Gamma].
$$

After base change by the atlas
\(\mathcal D\to[\mathcal D/\Gamma]\), this is a disjoint union of
\([\Gamma:\Lambda]\) copies of \(\mathcal D\).  It is therefore
representable, finite \'etale, and surjective.  Since \(\Lambda\) is
neat,
$$
  [\mathcal D/\Lambda]\simeq Y,
$$

which proves the proposition.
\end{proof}

\begin{corollary}[PEC on the neat product cover]
\label{cor:PEC-neat-product-cover}
With the notation above,

$$
  [\mathcal D/\Gamma]\text{ has PEC}
  \quad\Longleftrightarrow\quad
  Y\text{ has PEC}.
$$

\end{corollary}

\begin{proof}
The map

$$
  Y\longrightarrow[\mathcal D/\Gamma]
$$

is finite \'etale and surjective, so the assertion follows from
Proposition~\ref{prop:PEC-finite-etale}.
\end{proof}

\subsection{Toroidal compactifications and products}
\label{subsec:toroidal-products}

We finish the section by recording the compactifications of the two building
blocks and the product formula for their logarithmic cotangent bundles.

\begin{proposition}[Compactifications of the building blocks]
\label{prop:building-block-compactifications}
Let $X_i$ be a factor appearing in
proposition~\ref{prop:neat-product-cover}.  Then $X_i$ admits a smooth
projective compactification
\[
  X_i=\overline X_i\setminus D_i
\]
for which $D_i$ is a simple normal-crossings divisor.  More precisely:
\begin{enumerate}
  \item if $X_i$ is of arithmetic type (Shimura), one may choose a regular projective
  admissible rational polyhedral cone decomposition and take the associated
  toroidal compactification;
  \item if $X_i$ is of ball quotient type, one may take the smooth projective toroidal compactification obtained by filling in the cusps;
  \item if $X_i$ is compact, one takes
  $\overline X_i=X_i$ and $D_i=\varnothing$.
\end{enumerate}
\end{proposition}

\begin{proof}
For a Shimura factor, this is the toroidal compactification theorem of
Ash-Mumford-Rapoport-Tai \cite{AMRT}.  After refining an admissible rational
polyhedral cone decomposition, it may be taken regular and projective.
Regularity and neatness give smoothness and simple normal-crossings boundary, while projectivity of the decomposition gives projectivity of the compactification, see \cite[Chapter~III]{AMRT}.

For a finite-volume complex-ball quotient whose parabolic elements are
unipotent, the toroidal compactification is smooth and projective, with boundary a disjoint union of abelian varieties as in \cite[Theorem~1]{Mok}. The compact case corresponds to empty boundary.
\end{proof}

\begin{lemma}
\label{lem:product-log-nef}
For $1\leq i\leq r$, let $\overline X_i$ be a smooth projective variety and
let $D_i\subset\overline X_i$ be a simple normal-crossings divisor.  Put
\[
  \overline X=\prod_{i=1}^r\overline X_i,
  \qquad
  D=\sum_{i=1}^r p_i^{-1}(D_i),
\]
where $p_i\colon\overline X\to\overline X_i$ is the projection.  Then
$\overline X$ is smooth and projective, $D$ is a simple normal-crossings
divisor, and
\begin{equation}
  \Omega^1_{\overline X}(\log D)
  \simeq
  \bigoplus_{i=1}^r
  p_i^*\Omega^1_{\overline X_i}(\log D_i).
  \label{eq:product-log-cotangent}
\end{equation}
In particular, if every
$\Omega^1_{\overline X_i}(\log D_i)$ is nef, then
$\Omega^1_{\overline X}(\log D)$ is nef.
\end{lemma}

\begin{proof}
One can check it locally, or
\cite[Lemma~2.3]{ArapuraPatel}.
\end{proof}

\begin{corollary}
\label{cor:final-reduction}
To prove Theorem~\ref{thm:main}, it is enough to establish logarithmic
nefness for the following two classes:
\begin{enumerate}
  \item connected components of Shimura varieties at neat finite level;
  \item neat finite-volume complex-ball quotients with unipotent parabolic
  elements.
\end{enumerate}
More precisely, suppose that for every factor $X_i$ in
proposition\ref{prop:neat-product-cover} one can choose the compactification
of Proposition~\ref{prop:building-block-compactifications} so that
\[
  \Omega^1_{\overline X_i}(\log D_i)
\]
is nef.  Then $\mathcal X_\Gamma$ has the PEC property.
\end{corollary}

\begin{proof}
Let
\[
  X=X_1\times\cdots\times X_r
\]
be the product supplied by proposition~\ref{prop:neat-product-cover}.  Choose smooth projective toroidal
compactifications
\[
  X_i=\overline X_i\setminus D_i
\]
with nef logarithmic cotangent bundles, and put
\[
  \overline X=\prod_{i=1}^r\overline X_i,
  \qquad
  D=\sum_{i=1}^r p_i^{-1}(D_i).
\]
By Lemma~\ref{lem:product-log-nef},
\[
  \Omega^1_{\overline X}(\log D)
\]
is nef.  Theorem~\ref{thm:AP} therefore implies that $X$ has PEC.  By
Corollary~\ref{cor:PEC-neat-product-cover}, $\mathcal X_\Gamma$ has PEC as well.
\end{proof}

Theorems~\ref{thm:shimura-log-nef} and~\ref{thm:cadorel-ball} provide the two
factorwise nefness statements in Corollary~\ref{cor:final-reduction}.  Hence,
after this section, no further stack-theoretic or group-theoretic reduction
is needed.  The remaining geometric verification takes place entirely on
smooth neat quotients. So the boundary analysis reduces precisely to smooth noncompact connected components of finite-level Shimura varieties and smooth noncompact finite-volume complex-ball quotients.

\section{Nef logarithmic cotangent bundle}\label{sec:nefness}

\subsection{Arithmetic quotients}
\label{sec:arithmetic-quot}

Let \(X=\Gamma\backslash\mathcal D\) be a connected component of a
Shimura variety at neat finite level, and let
\[
  j\colon X\hookrightarrow\bar X
\]
be a smooth projective toroidal compactification with
\[
  D=\bar X\setminus X
\]
a simple normal-crossings divisor.
Fix a point $h \in \mathcal{D}$, we can write $\mathcal{D} = G/K_h = \mathbf{G}(\mathbb{R})^+/K_h$ where $\mathbf{G}$ is an adjoint rational algebraic group. Let $\mathbb{S} = \mathrm{Res}_{\mathbb{C}/\mathbb{R}} \bb{G}_m$ be the Deligne torus, then equivalently, one can view $\mathcal{D}$ as a $\mathbf{G}(\mathbb{R})^+$-conjugacy class of a homomorphism $h: \mathbb{S} \to \mathbf{G}_{\mathbb{R}}$. The adjoint complexified representation $\Rm{Ad} \circ h: \bb{S}_{\bb{C}} \to \Rm{Gl}(\fr{g}_{\bb{C}})$ induces a decomposition 
\[\fr{g}_{\bb{C}} = \fr{g}^{-1, 1}_h \oplus \fr{g}^{0, 0}_h \oplus \fr{g}^{1, -1}_h\]
with the convention that $\Rm{Ad}(h(z))$ acts by $z^{-p}\bar{z}^{p}$ on $\fr{g}_h^{p, -p}$. We can define the Hodge filtration 
\[F^{1}_h \fr{g}_{\bb{C}} = \fr{g}^{1, -1}_h, \tab F^0_h \fr{g}_{\bb{C}} = \fr{g}_h^{1, -1} \oplus \fr{g}_h^{0, 0}, \tab F^{-1}_h \fr{g}_{\bb{C}} = \fr{g}_{\bb{C}}\]
\tab For $x = g \cdot h \in \mathcal{D}$, the corresponding representation $h_x: \bb{S} \to \Rm{GL}(\fr{g}_{\bb{C}})$ is obtained from $h$ via conjugation by $g$. We then have the analogous decomposition $\fr{g}_{\bb{C}} = \fr{g}^{-1, 1}_x \oplus \fr{g}^{0, 0}_x \oplus \fr{g}^{1, -1}_x$ and filtration $F^\bullet_x \fr{g}_{\bb{C}}$. 
\subsubsection{The adjoint polarized variation}

Let $P_h = \{g \in \mathbf{G}_{\bb{C}}\vert\ \Rm{Ad}(g)F^p_h = F^p_h \text{ for all }p\}$ be the stabilizer of the filtration $F^\bullet_h$. Its Lie algebra is $\fr{p}_h = F^0_{h}\fr{g}_{\bb{C}}$, and we have the compact dual 
\[\check{\mathcal{D}} = \mathbf{G}_{\bb{C}}/P_h\]
which is a flag variety. The Borel embedding 
\[\beta: \mathcal{D} \to \check{\mathcal{D}}, \tab x \mapsto F_x^\bullet\]
is a $\mathbf{G}(\bb{R})^+$-equivariant holomorphic embedding which identifies $\mathcal{D}$ with an open orbit in $\check{\mathcal{D}}$. \\
\tab
Given \(W\) a finite-dimensional algebraic representation of \(P_h\), we can define a homogeneous vector bundle on \(\check{\sh{D}}\)
\[
  \check{\sh{W}}
  =
  \Bf{G}_{\bb{C}}\times^{P_h}W
  =
  (\Bf{G}_{\bb{C}}\times W)/
  \bigl((g,w)\sim(gg',(g')^{-1}w)\bigr),
\]
which is \(\Bf{G}_{\bb{C}}\)-equivariant. Let
\[
  \check{\sh{V}}
  =
  \Bf{G}_{\bb{C}}\times^{P_h}\fr{g}_{\bb{C}},
  \qquad
  \check{\sh{F}}^p
  =
  \Bf{G}_{\bb{C}}\times^{P_h}F_h^p\fr{g}_{\bb{C}}.
\]
Since the adjoint representation of \(P_h\) on
\(\fr{g}_{\bb{C}}\) extends to \(\Bf{G}_{\bb{C}}\), there is a
\(\Bf{G}_{\bb{C}}\)-equivariant isomorphism
\[
  \check{\sh{V}}
  \simeq
  \check{\sh{D}}\times\fr{g}_{\bb{C}},
  \qquad
  [g,v]\longmapsto
  \bigl(gP_h,\Rm{Ad}(g)v\bigr).
\]
Hence
\(
  \beta^*\check{\sh{V}}
  \simeq
  \sh{D}\times\fr{g}_{\bb{C}},
\)
and under this identification the \(\Gamma\)-action is
\[
  \gamma\cdot(x,v)
  =
  \bigl(\gamma x,\Rm{Ad}(\gamma)v\bigr).
\]
Thus the adjoint representation
\(\Rm{Ad}|_\Gamma:\Gamma\to\Rm{GL}(\fr{g}_{\bb{Q}})\)
defines the rational local system
\[
  \bb{V}_{\bb{Q}}
  =
  \Gamma\backslash
  \bigl(\sh{D}\times\fr{g}_{\bb{Q}}\bigr).
\]

The bundles
\(\beta^*\check{\sh{V}}\) and
\(\beta^*\check{\sh{F}}^p\) are
\(\Bf{G}(\bb{R})^+\)-equivariant, hence \(\Gamma\)-equivariant, and
therefore descend to an automorphic vector bundle \(\sh{V}\) on \(X\),
together with a filtration by automorphic subbundles
\(\sh{F}^p\subset\sh{V}\). Moreover,
\(
  \sh{V}
  \simeq
  \bb{V}_{\bb{Q}}\otimes_{\bb{Q}}\sh{O}_X.
\)
If
\[
  \pi\colon\sh{D}\longrightarrow X=\Gamma\backslash\sh{D}
\]
denotes the quotient map, then by construction
\[
  \pi^*\sh{V}
  \simeq
  \beta^*\check{\sh{V}},
  \qquad
  \pi^*\sh{F}^p
  \simeq
  \beta^*\check{\sh{F}}^p.
\]
Let \(\nabla\) be the flat connection on \(\sh{V}\) obtained by
descending the exterior derivative \(1\otimes\Rm{d}\). Then
\[
  (\bb{V}_{\bb{Q}},\sh{F}^{\bullet},\nabla)
\]
is a weight-zero rational variation of Hodge structures; see
\cite[Proposition~5.9]{MilneShimura}.
\\
\tab 
Now let $\kappa(v,w)=\operatorname{Tr}(\operatorname{ad}(v)\operatorname{ad}(w))$ be the Killing form on \(\mathfrak g_{\mathbb Q}\), and put
\(Q=-\kappa\). Then \(Q\) is rational, symmetric, nondegenerate, and
\(\operatorname{Ad}(\mathbf G)\)-invariant. Thus it descends to a flat
bilinear form on \(\mathbb V_{\mathbb Q}\).

The invariance of \(Q\) under \(h_x(\mathbb S)\) implies
\[
Q\bigl(\mathfrak g_x^{p,-p},\mathfrak g_x^{q,-q}\bigr)=0
\qquad\text{unless }p+q=0.
\]
Indeed, if \(v\in\mathfrak g_x^{p,-p}\) and
\(w\in\mathfrak g_x^{q,-q}\), then for every
\(z\in\mathbb C^\times\),
\[
Q(v,w)
=
z^{-(p+q)}\overline z^{\,p+q}Q(v,w),
\]
which forces \(Q(v,w)=0\) unless \(p+q=0\).
Since \(Q\) is nondegenerate, it therefore gives perfect pairings
\[
\mathfrak g_x^{p,-p}
\times
\mathfrak g_x^{-p,p}
\longrightarrow\mathbb C.
\]

Moreover, by the characterization of a Hermitian symmetric domain,
\[
C_x=\operatorname{Ad}(h_x(i))
\]
is a Cartan involution of \(\mathfrak g_{\mathbb R}\). Hence the real
symmetric form
\[
(v,w)\longmapsto Q(v,C_xw)=-\kappa(v,C_xw)
\]
is positive definite on \(\mathfrak g_{\mathbb R}\).
Equivalently, the associated Hodge Hermitian form is positive definite.
Thus \(Q\) polarizes the weight-zero variation of Hodge structure $(\bb{V}_{\bb{Q}},\sh{F}^\bullet,\nabla)$. \\
\\
\tab Now take the associated graded to get a system of Higgs bundles 
\[(E, \theta) = \Rm{gr}_{\sh{F}}(\sh{V}, \nabla), \tab \theta: E^{p, -p} \to E^{p - 1, - p + 1} \otimes \Omega^1_X\]
and $E^{-1, 1} = \Gamma\setminus (\beta^*(\Bf{G}_{\bb{C}} \times^{P_h} \Rm{gr}_{F_h}^{-1} \fr{g}_{\bb{C}}))$. At $h \in \sh{D} \subset \check{\sh{D}}$, the tangent space is 
\[T_h \sh{D} = T_h \check{\sh{D}} \simeq \fr{g}_{\bb{C}}/\fr{p}_h = \fr{g}_{\bb{C}}/F^0_h\fr{g}_{\bb{C}} \simeq \fr{g}_h^{-1, 1}\]
and the same is true for any $x \in \sh{D}$. These tangent identifications are $\Bf{G}(\bb{R})^+$-equivariant and holomorphically varying, hence $T_{\sh{D}}^{\Rm{hol}} \simeq \beta^*(\Bf{G}_{\bb{C}} \times^{P_h} \Rm{gr}_{F_h}^{-1} \fr{g}_{\bb{C}})$. Quotienting by $\Gamma$ we get $T_X \simeq E^{-1, 1}$. Also, $\theta(T_X) = \theta(E^{-1, 1}) = 0$ since there is no $(-2, 2)$-graded piece.  

\subsubsection{Deligne and Mumford canonical extensions}

We have an exact functor called canonical extension
$$
  (-)^{\mathrm{can}}\colon
  \{\text{automorphic vector bundles on }X\}
  \longrightarrow
  \{\text{vector bundles on }\bar X\};
$$
see \cite[\S3, Main Theorem~3.1]{mumford1977hirzebruch} and
\cite[\S4]{harris1989functorial}. The functor is compatible with
tensor products and duals, and hence also with internal Hom.
In particular, canonical extension preserves exact sequences of
automorphic vector bundles.

The adjoint variation $(\bb{V}_{\bb{Q}}, \sh{F}^\bullet, \nabla, Q)$ has unipotent local monodromy around $D$ \cite[proof of Theorem 4.2]{harris1989functorial}, so we have Deligne's canonical extension \cite{deligne2006equations} $(\bar{\sh{V}}, \bar{\nabla})$ which is a logarithmic flat bundle on $(\bar{X}, D)$, i.e., $\bar{V}$ is locally free over $\sh{O}_{\bar{X}}$ and 
\[\bar{\nabla}: \bar{\sh{V}} \to \bar{\sh{V}} \otimes \Omega^1_{\bar{X}}(\log D)\]
\tab
By Harris's comparison theorem, the Mumford canonical extension of the
flat automorphic bundle \(\mathcal V\) agrees with Deligne's canonical
extension; see \cite[Theorem~4.2]{harris1989functorial}. 

For each \(p\), let
\[
  \mathcal A^p
  :=
  (\mathcal F^p)^{\mathrm{can}}
\]
denote the automorphic canonical extension of \(\mathcal F^p\).
Since
\[
  0\longrightarrow\mathcal F^p
  \longrightarrow\mathcal V
  \longrightarrow\mathcal V/\mathcal F^p
  \longrightarrow0
\]
is an exact sequence of automorphic vector bundles, exactness of the
canonical-extension functor gives
\[
  0\longrightarrow\mathcal A^p
  \longrightarrow\bar{\mathcal V}
  \longrightarrow
  (\mathcal V/\mathcal F^p)^{\mathrm{can}}
  \longrightarrow0.
\]
In particular, the quotient is locally free, so
\(\mathcal A^p\subset\bar{\mathcal V}\) is saturated.

We claim that, inside \(j_*\mathcal V\),
\[
  \mathcal A^p
  =
  (j_*\mathcal F^p)\cap\bar{\mathcal V}.
\]
Indeed, the inclusion from left to right follows from
\(\mathcal A^p|_X=\mathcal F^p\). Conversely, a local section of
\(\bar{\mathcal V}\) whose restriction to \(X\) lies in
\(\mathcal F^p\) maps to a section of
\(\bar{\mathcal V}/\mathcal A^p\) which vanishes on the dense open
subset \(X\). Since this quotient is locally free, that section
vanishes identically. Hence
\[
  (\mathcal F^p)^{\mathrm{can}}
  =
  (j_*\mathcal F^p)\cap\bar{\mathcal V}=\bar{\mathcal F}^p.
\]
Thus the automorphic canonical extension of \(\mathcal F^p\) agrees
with the extended Hodge filtration inside Deligne's canonical extension
\(\bar{\mathcal V}\). In particular,
\(\bar{\mathcal F}^p\subset\bar{\mathcal V}\) is a locally free
saturated subsheaf.

Applying exactness of the canonical-extension functor to
\[
  0\longrightarrow\mathcal F^{p+1}
  \longrightarrow\mathcal F^p
  \longrightarrow E^{p,-p}
  \longrightarrow0
\]
gives an exact sequence
\[
  0\longrightarrow\bar{\mathcal F}^{p+1}
  \longrightarrow\bar{\mathcal F}^p
  \longrightarrow
  (E^{p,-p})^{\mathrm{can}}
  \longrightarrow0.
\]
Hence
\[
  \bar{\mathcal F}^p/\bar{\mathcal F}^{p+1}
  \simeq
  (E^{p,-p})^{\mathrm{can}}.
\]
In particular, the graded pieces are locally free.

The logarithmic connection satisfies Griffiths transversality:
$$
  \bar{\nabla}(\bar{\mathcal F}^p)
  \subseteq
  \bar{\mathcal F}^{p-1}
  \otimes
  \Omega^1_{\bar X}(\log D).
$$

Taking the associated graded therefore gives a logarithmic Higgs bundle

$$
  (\bar E,\bar\theta)
  =
  \operatorname{Gr}_{\bar{\mathcal F}}
  (\bar{\mathcal V},\bar\nabla),
$$

where

$$
  \bar E^{p,-p}
  =
  \bar{\mathcal F}^p/\bar{\mathcal F}^{p+1},
  \qquad
  \bar\theta\colon
  \bar E^{p,-p}
  \longrightarrow
  \bar E^{p-1,-p+1}
  \otimes
  \Omega^1_{\bar X}(\log D).
$$

\begin{proposition}
\label{prop:boundary}
There are canonical isomorphisms

$$
  \bar E^{-1,1}\simeq T_{\overline X}(-\log D),
  \qquad
  \bar F^1\simeq\Omega^1_{\overline X}(\log D)
  \simeq(\bar E^{-1,1})^\vee.
$$

\end{proposition}

\begin{proof}
Recall that on \(X\) we have the canonical identification
\(E^{-1,1}\simeq T_X.\)
By the discussion above,
\[
  \bar E^{-1,1}
  =
  \operatorname{Gr}^{-1}_{\bar{\mathcal F}}\bar{\mathcal V}
  \simeq
  (E^{-1,1})^{\mathrm{can}}
  \simeq
  (T_X)^{\mathrm{can}}.
\]
By \cite[Proposition~3.4(a)]{mumford1977hirzebruch},
\[
  (\Omega_X^1)^{\mathrm{can}}
  \simeq
  \Omega^1_{\bar X}(\log D).
\]
Since canonical extension commutes with duals, it follows that
\[
  \bar E^{-1,1}
  \simeq
  (T_X)^{\mathrm{can}}
  \simeq
  T_{\bar X}(-\log D).
\]
Dualizing gives
\(
  (\bar E^{-1,1})^\vee
  \simeq
  \Omega^1_{\bar X}(\log D).
\)
On \(X\), the polarization induces an automorphic isomorphism
\(
  E^{1,-1}
  \simeq
  (E^{-1,1})^\vee.
\)
Since canonical extension is compatible with duals, we obtain
\[
  \bar E^{1,-1}
  \simeq
  (\bar E^{-1,1})^\vee.
\]
Finally, since \(\mathcal F^2=0\), its canonical extension is also zero,
so
\(
  \bar{\mathcal F}^2=0.
\)
Therefore
\[
  \bar{\mathcal F}^1
  =
  \bar{\mathcal F}^1/\bar{\mathcal F}^2
  =
  \bar E^{1,-1}.
\]

Combining these identifications yields
\[
  \bar E^{-1,1}
  \simeq
  T_{\bar X}(-\log D),
  \qquad
  \bar{\mathcal F}^1
  \simeq
  \Omega^1_{\bar X}(\log D)
  \simeq
  (\bar E^{-1,1})^\vee.
\]
\end{proof}

\begin{theorem}[Arithmetic nefness]
\label{thm:arithmetic-nef}
Let \(X=\Gamma\backslash\mathcal D\) be a quotient arising from
a neat arithmetic subgroup of a rational group of Hermitian type. If
\(j\colon X\hookrightarrow\overline X\) is a smooth projective toroidal
compactification whose reduced boundary
\(D=\overline X\setminus X\) has simple normal crossings, then
$$
  \Omega^1_{\overline X}(\log D)
  \quad\text{is nef}.
$$

\end{theorem}

\begin{proof}
The bundle \(\bar E^{-1,1}\) is a holomorphic subbundle of
\(\bar E\), and \(\bar\theta(\bar E^{-1,1})=0\)
since there is no \(\bar E^{-2,2}\). Thus
\cite[Theorem~1.8]{brunebarbe2018symmetric} implies that
\((\bar E^{-1,1})^\vee\) is nef. By
Proposition~\ref{prop:boundary},
\[
  (\bar E^{-1,1})^\vee
  \simeq
  \Omega^1_{\bar X}(\log D),
\]
which proves the theorem.
\end{proof}

\begin{remark}
Alternatively, since \(\sh{F}^1\) is the lowest nonzero piece of the Hodge
filtration, we may apply
\cite[Corollary~1.2]{fujino2019semipositivity} to
$$
  \bar{\sh{F}}^1\simeq\Omega^1_{\bar X}(\log D)
$$

to obtain the desired nefness in this way as well.
\end{remark}

\subsection{Ball quotients and completion of the argument}
The arithmetic ball quotient case follows from the previous argument, while the nonarithematic case is a direct application of the following proposition:
\begin{proposition}\cite[Theorem~3]{Cadorel}\label{prop:ball-nef}
Let
\[
  X=\Gamma\backslash\mathbb B^n,
\]
where \(\Gamma\subset\operatorname{Aut}(\mathbb B^n)\) is a neat
finite-covolume lattice whose parabolic elements are unipotent. Let
\[
  X\hookrightarrow\overline X
\]
be its smooth projective toroidal compactification and let
\(D=\overline X\setminus X\) be the reduced boundary divisor. Then
\[
  \Omega^1_{\overline X}(\log D)
  \quad\text{is nef}.
\]
\end{proposition}

Thus all factors occurring in
proposition~\ref{prop:neat-product-cover}
admit smooth projective logarithmic compactifications with nef
logarithmic cotangent bundle. By \ref{lem:product-log-nef} we get the nefness for the product as well.

\section{Perverse Euler characteristics}
\label{sec:Euler-characteristics}

In order to use nefness results to get our desired inequalities, we first recall some facts which are discussed in \cite{ArapuraPatel} and \cite{ArapuraMesePatel}. We will state it in the following generality.

Let $\overline{X}$ be a smooth projective variety of dimension $n$ with $D\subset \overline{X}$ an SNC divisor. We denote $X=\overline{X}\setminus D$ for the quasi-projective variety. 
For a constructible complex $K$ characteristic cycle $CC(K)\subset T^*X$ is a conical Lagrangian cycle of dimension $n$, see \cite[Chapter~IX]{KashiwaraSchapira} and \cite[Sections~4.1--4.2]{Dimca}. If $P$ is perverse, then $CC(P)$ is effective. Moreover, if $P$ has generic rank $r$ then
\[CC(P)= r[T_X^*X]+\sum_{\alpha}m_{\alpha}[\Lambda_{\alpha}], \qquad m_{\alpha}>0,\] where $\Lambda_\alpha$ are the remaining irreducible conical components.
Let $\overline{CC(P)}$ denote the closure of $CC(P)$ in $\Omega^1_{\overline{X}}(\log D)$. The logarithmic Dubson-Kashiwara index formula of Wu-Zhou \cite[Theorem~1.6]{WuZhou} gives \[ \chi(X,P) = [\overline{CC(P)}]\cdot[s_0(\overline{X})] \] where $s_0:\overline{X} \to \Omega^1_{\overline{X}}(\log D)$ is the zero section.

If $\Omega^1_{\overline{X}}(\log D)$ is nef, positivity of conical cycles in a nef vector bundle \cite[Example~8.2.6]{LazarsfeldPositivityII} shows that the intersection of each irreducible conical component with the zero section is nonnegative. This gives the criterion stated in \ref{thm:AP} and gives $\chi(X,P)\geq 0$ in the nef log cotangent bundle case.

One can improve this result and get a sharper inequality as shown in \cite{ArapuraMesePatel}.

\begin{proposition}[Arapura--Mese--Patel \cite{ArapuraMesePatel}] 
\label{prop:AMP-positivity} 

 Let $\overline{X}$ be a smooth projective variety of dimension $n$ with $D\subset \overline{X}$ an SNC divisor and $X:=\overline{X}\setminus D$. Suppose that $\Omega_{\overline{X}}^1(\log D)$ is nef, let $P\in\Perv(X,\Q)$ have generic rank $r$. Then \[ \chi(X,P)\geq r(-1)^n\chi(X). \] Consequently, if $P$ has full support and $\chi(X)\neq 0$, then \[ \chi(X,P)>0. \] 
\end{proposition}

Hence, in the full support case, strict positivity reduces to showing that $\chi(X)\neq 0$. We have already established the required nefness statements for locally Hermitian symmetric spaces. We first deduce the corresponding nonnegativity results, and then verify that $\chi(X)\neq 0$ in the cases under consideration, which yields strict positivity.

\subsection{Locally Hermitian symmetric varieties}

Now we will prove the Euler characteristics result by combining the criterion given in \ref{thm:AP} and the nefness results obtained above.

\begin{theorem}[Locally Hermitian symmetric varieties]
\label{thm:smooth-pec}
Let \(\mathcal D\) be a Hermitian symmetric domain of noncompact type
and let $\Gamma\subset\operatorname{Aut}(\mathcal D)^\circ$ be a torsion-free lattice. Define $X_\Gamma=\Gamma\backslash\mathcal D.$
Then \(X_\Gamma\) has the PEC property i.e.
\[
  \chi(X_\Gamma,P)\geq0
  \qquad
  \text{for every }P\in\Perv(X_\Gamma,\Q).
\]
In particular, if \(n=\dim_{\C}X_\Gamma\), then
\[
  (-1)^n\chi(X_\Gamma)\geq0.
\]
\end{theorem}

\begin{proof}
By Proposition~\ref{prop:neat-product-cover}, there is a finite index
neat subgroup $\Lambda=\Lambda_1\times\cdots\times\Lambda_r\subset\Gamma$
such that
$$
  Y=\Lambda\backslash\mathcal D
   =Y_1\times\cdots\times Y_r
$$
is a smooth quasi-projective variety and
$$
  Y\longrightarrow X_\Gamma
$$

is finite \'etale and surjective.  Each \(Y_i\) is either an arithmetic
locally symmetric variety or a finite-volume complex-ball quotient with
unipotent parabolic elements.

Choose for each \(Y_i\) the smooth projective compactification $Y_i=\overline Y_i\setminus D_i$ of Proposition~\ref{prop:building-block-compactifications}.  By
Theorem~\ref{thm:arithmetic-nef} and
Proposition~\ref{prop:ball-nef}, respectively,
$$
  \Omega^1_{\overline Y_i}(\log D_i)
$$
is nef for every \(i\).  Hence Lemma~\ref{lem:product-log-nef} shows
that, for
$$
  \overline Y=\prod_i\overline Y_i,
  \qquad
  D=\overline Y\setminus Y,
$$
the log cotangent bundle $\Omega^1_{\overline Y}(\log D)$ is nef.  Theorem~\ref{thm:AP} therefore gives PEC for \(Y\).
Since \(Y\to X_\Gamma\) is finite \'etale and surjective,
Proposition~\ref{prop:PEC-finite-etale} gives PEC for \(X_\Gamma\).
Finally, \(\Q_{X_\Gamma}[n]\) is perverse because \(X_\Gamma\) is smooth
of dimension \(n\). Thus
$$
  0\leq
  \chi\bigl(X_\Gamma,\Q_{X_\Gamma}[n]\bigr)
  =
  (-1)^n\chi(X_\Gamma).
$$

\end{proof}

\subsection{Locally Hermitian symmetric Deligne--Mumford stacks}

We now recover the stack statement formally from the neat-cover
reduction.

\begin{theorem}[The stack case]
\label{thm:stack-pec}
Let \(\mathcal D\) be a Hermitian symmetric domain of noncompact type
and let $\Gamma\subset\operatorname{Aut}(\mathcal D)^\circ$
be a lattice. Let \(\mathcal X_\Gamma\) be the algebraic
Deligne-Mumford stack constructed above, whose analytification is
$\mathcal X_\Gamma^{\mathrm{an}}\simeq[\mathcal D/\Gamma].$
Then
\[
  \chiorb(\mathcal X_\Gamma,P)\geq0
  \qquad
  \text{for every }P\in\Perv(\mathcal X_\Gamma,\Q).
\]
In particular, if \(n=\dim_{\C}\mathcal D\), then
\[
  (-1)^n\chiorb(\mathcal X_\Gamma)\geq0.
\]
\end{theorem}

\begin{proof}
Let $Y=\Lambda\backslash\mathcal D$ be the smooth quasi-projective variety supplied by Proposition~\ref{prop:neat-product-cover}.  By
Theorem~\ref{thm:smooth-pec}, \(Y\) has PEC.  Moreover, Proposition
~\ref{prop:neat-product-cover} gives a representable finite \'etale
surjective morphism of analytic Deligne--Mumford stacks
$$
  Y^{\mathrm{an}}
  \longrightarrow
  [\mathcal D/\Gamma]
  \simeq
  \mathcal X_\Gamma^{\mathrm{an}}.
$$

Hence Proposition~\ref{prop:PEC-finite-etale} implies that
\(\mathcal X_\Gamma\) has PEC.

Finally, \(\mathcal X_\Gamma\) is smooth of dimension
\(n=\dim_{\C}\mathcal D\), so
\(\Q_{\mathcal X_\Gamma}[n]\) is perverse. Therefore
$$
  0\leq
  \chiorb\bigl(
    \mathcal X_\Gamma,
    \Q_{\mathcal X_\Gamma}[n]
  \bigr)
  =
  (-1)^n\chiorb(\mathcal X_\Gamma).
$$

\end{proof}

We also record the following corollary for finite level Shimura stacks.
\begin{corollary}[Finite-level Shimura stacks]
\label{cor:shimura-full}
Let \((G,\mathcal X)\) be a Shimura datum and let
\(K\subset G(\A_f)\) be an arbitrary compact open subgroup. Then
\[
  \chiorb\bigl(
    \Sh_K^{\mathrm{DM}}(G,\mathcal X)_{\C},P
  \bigr)\geq0
\]
for every
\[
  P\in
  \Perv\bigl(
    \Sh_K^{\mathrm{DM}}(G,\mathcal X)_{\C},\Q
  \bigr).
\]
If \(K\) is neat, then
\(\Sh_K^{\mathrm{DM}}(G,\mathcal X)_{\C}\)
is represented by the usual smooth finite-level Shimura variety, and
\(\chiorb\) is the ordinary Euler characteristic.
\end{corollary}

\begin{proof}
Choose a normal neat compact open subgroup$K'\triangleleft K.$ Then
$$
  \Sh_K^{\mathrm{DM}}(G,\mathcal X)
  =
  \bigl[
    \Sh_{K'}(G,\mathcal X)/(K/K')
  \bigr],
$$

so the natural morphism
$$
  \Sh_{K'}(G,\mathcal X)_{\C}
  \longrightarrow
  \Sh_K^{\mathrm{DM}}(G,\mathcal X)_{\C}
$$
is finite \'etale and surjective.

Since \(K'\) is neat,
\(\Sh_{K'}(G,\mathcal X)_{\C}\) is a smooth quasi-projective variety,
and its connected components are Hermitian locally symmetric varieties.
Hence it has PEC by Theorem~\ref{thm:smooth-pec}.  Therefore
Proposition~\ref{prop:PEC-finite-etale} implies that
\(\Sh_K^{\mathrm{DM}}(G,\mathcal X)_{\C}\) has PEC.

If \(K\) is neat, we may take \(K'=K\), so
\(\Sh_K^{\mathrm{DM}}(G,\mathcal X)_{\C}\) is the usual smooth Shimura
variety and \(\chiorb\) is the ordinary Euler characteristic.
\end{proof}

\subsection{Positivity of Euler Characteristics}
To get strict inequality for $P \in \Rm{Perv}(\sh{X}_\Gamma, \bb{Q})$ with full support, we first have a nonvanishing result:
\begin{proposition}
\label{prop:Euler-nonvanishing}
    Let $\sh{D}$ be a Hermitian symmetric domain of noncompact type, and $\Gamma \subset \Rm{Aut}(\sh{D})^\circ$ be a torsion-free lattice. Put $X_\Gamma = \Gamma\setminus \sh{D}$. Then $\chi(X_{\Gamma}) \neq 0$. 
\end{proposition}
\begin{proof}
    Euler characteristic is multiplicative, and its nonvanishing is preserved by finite covers, so by proposition \ref{prop:neat-product-cover}, it suffices to assume $X_{\Gamma}$ falls into one of the two cases: $\sh{D}$ is an irreducible Hermitian symmetric domain of dimension $n$ with $\Gamma$ an arithmetic neat lattice, or $\sh{D} = \bb{B}^n$ with $\Gamma$ a neat lattice whose parabolic elements are unipotent. Let $(\overline{X}_{\Gamma}, D)$ be the toroidal compactification. \\
    \tab We claim that $\chi(X_{\Gamma}) \neq 0$ is equivalent to $\chi(\check{D}) \neq 0$ where $\check{\sh{D}}$ is the compact dual of $\sh{D}$.  In the arithmetic case, using the notations of section \ref{sec:arithmetic-quot}, we have $E^{-1, 1} \simeq T_{X_{\Gamma}}$ on $X_{\Gamma}$ with the corresponding vector bundle $\mathbf{G}_{\bb{C}} \times^{P_h} \Rm{gr}_{F_h}^{-1} \fr{g}_{\bb{C}} \simeq T_{\check{\sh{D}}}$ on the compact dual. We also have the Mumford canonical extension $\overline{E}^{-1, 1} \simeq T_{\overline{X}_{\Gamma}}(-\log D)$. By \cite[Theorem 3.2]{mumford1977hirzebruch} we have 
    \[ \int_{\overline{X}_{\Gamma}} c_n(T_{\overline{X}_{\Gamma}}(-\log D)) = (-1)^n \cdot K \cdot \int_{\check{\sh{D}}} c_n(T_{\check{\sh{D}}})
    \]
    with $K$ a positive constant. The same proportionality result is true in the nonarithmetic ball quotient case as well \cite[\S2]{hwang2004volumes}. Thus, by (logarithmic) Gauss-Bonnet \cite[Theorem 3]{norimatsu1978kodaira}, 
    \[ \chi(X_{\Gamma}) = (-1)^n \cdot K \cdot \chi(\check{\sh{D}})
    \]
    \tab Finally, $\chi(\check{\sh{D}}) \neq 0$ since, by the Bruhat decomposition, $\check{\sh{D}} = \mathbf{G}_{\bb{C}}/P_h$ is a finite CW-complex with only even $\bb{R}$-dimensional cells. 
\end{proof}

\begin{corollary}
    Let $\sh{D}$ be a Hermitian symmetric domain of noncompact type, and $\Gamma \subset \Rm{Aut}(\sh{D})^\circ$ be a lattice. Denote $\sh{X}_{\Gamma} = [\Gamma\setminus \sh{D}]$ the Deligne-Mumford stack. If $P \in \Rm{Perv}(\sh{X}_{\Gamma}, \bb{Q})$ has full support, then 
    \[\chi^{\Rm{orb}}(\sh{X}_{\Gamma}, P) > 0\]
\end{corollary}
\begin{proof}
    By \cite[Theorem 2.5(iii)]{ArapuraPatel}, we can assume that $\Gamma$ is neat and it suffices to show that $\chi(X_{\Gamma}, P) > 0$ for all perverse $P$ with full support. This follows from proposition \ref{prop:AMP-positivity} and proposition \ref{prop:Euler-nonvanishing}. 
\end{proof}

\section{Applications}

\subsection{Properties}

For applications we will need the following properties:

\begin{proposition}
\label{prop:PEC-permanence}
Let \(X\) be a complex algebraic variety having PEC. Then the following
statements hold.
\begin{enumerate}
\item Every closed subvariety
$$
  i\colon Z\hookrightarrow X
$$
has PEC.
\item Every locally closed subvariety
\[
  j\colon U\hookrightarrow X
\]
for which \(j\) is affine has PEC. In particular, this applies to the
complement of an effective Cartier divisor.
\item Let \(U\) be either \(X\) or a subvariety occurring in
\textup{(1)} or \textup{(2)}. If
$$
  f\colon Y\longrightarrow U
$$
is finite, then \(Y\) has PEC.
\end{enumerate}
\end{proposition}

\begin{proof}

\begin{enumerate}
    \item Consider \(P\in\Perv(Z,\Q)\), then
\(Ri_*P=i_*P\) is perverse and $R\Gamma(X,i_*P)\simeq R\Gamma(Z,P).$
Hence $\chi(Z,P)=\chi(X,i_*P)\geq0.$
    \item Let \(P\in\Perv(U,\Q)\). Since \(j\) is affine and quasi-finite, \(Rj_*P\) is perverse. Moreover, $R\Gamma(X,Rj_*P)\simeq R\Gamma(U,P).$ Therefore$\chi(U,P)=\chi(X,Rj_*P)\geq0.$
    \item Let \(P\in\Perv(Y,\Q)\). Since \(f\) is finite, \(Rf_*P\) is perverse, and $R\Gamma(U,Rf_*P)\simeq R\Gamma(Y,P).$ Since \(U\) has PEC by \textup{(1)} or \textup{(2)}, as appropriate, we obtain $\chi(Y,P)=\chi(U,Rf_*P)\geq0.$
\end{enumerate}
The required perverse \(t\)-exactness statements follow from
\cite[Corollaire~4.1.3]{BBD82}.

\end{proof}

One can upgrade this property to stacks but we do not need it in our examples. For moduli space applications we also record the following corollary.

\begin{corollary}[Arrangement complements]
\label{cor:PEC-arrangement-complement}
Let
$\mathcal X=[\mathcal D/\Gamma]$
be a finite-volume locally Hermitian symmetric Deligne--Mumford stack,
and let $\mathcal H\subset\mathcal D$ be a \(\Gamma\)-invariant locally finite union of hypersurfaces. Suppose
that there is a neat finite-index subgroup $\Lambda\subset\Gamma$
such that $X_\Lambda=\Lambda\backslash\mathcal D$
is a complex algebraic variety and the image $H_\Lambda=\Lambda\backslash\mathcal H \subset X_\Lambda$
is an algebraic effective Cartier divisor. Then
\[
  [(\mathcal D\setminus\mathcal H)/\Gamma]
\]
has PEC.
\end{corollary}

\begin{proof}
By Theorem~\ref{thm:smooth-pec}, the variety \(X_\Lambda\) has PEC.
Since \(H_\Lambda\) is an effective Cartier divisor, the open immersion
\[
  X_\Lambda\setminus H_\Lambda
  \hookrightarrow
  X_\Lambda
\]
is affine. Hence $X_\Lambda\setminus H_\Lambda$
has PEC by Proposition~\ref{prop:PEC-permanence}\textup{(2)}.

The natural morphism
\[
  X_\Lambda\setminus H_\Lambda
  \longrightarrow
  [(\mathcal D\setminus\mathcal H)/\Gamma]
\]
is representable, finite \'etale, and surjective. The conclusion follows
from Proposition~\ref{prop:PEC-finite-etale}.
\end{proof}

\subsection{Moduli of smooth cubic threefolds}

Let $V=\operatorname{Sym}^{3}(\C^{5})^{\vee}$, $
  G=\operatorname{PGL}_{5},$ and
$$
  \mathfrak M_{0}
  =
  \bigl[\mathbb P(V)^{\mathrm{sm}}/G\bigr]
$$

be the Deligne-Mumford stack of smooth cubic threefolds. Allcock--Carlson--Toledo associate to this moduli problem a complex ball
\(\mathbb B^{10}\), an arithmetic group \(P\Gamma\), and two locally
finite \(P\Gamma\)-invariant hyperplane arrangements
\(\mathcal H_c\) and \(\mathcal H_\Delta\), called the chordal and
discriminant arrangements. They construct the moduli space of smooth
cubic threefolds as a complex analytic orbifold; see \cite[\S2]{ACT}. Their period map identifies this orbifold
with its image in the ball quotient, and
\cite[Theorem~7.1]{ACT} identifies that image with the complement of
the two arrangements. Thus there is an isomorphism of analytic
Deligne--Mumford stacks
$$
  \mathfrak M_0^{\mathrm{an}}
  \simeq
  \left[
    \bigl(
      \mathbb B^{10}
      \setminus
      (\mathcal H_c\cup\mathcal H_\Delta)
    \bigr)/P\Gamma
  \right].
$$

\begin{corollary}
\label{cor:cubic-threefold-pec}
The Deligne--Mumford stack \(\mathfrak M_0\) has PEC, i.e.
$$
  \chiorb(\mathfrak M_0,P)\geq0
  \qquad
  \text{for every }P\in\Perv(\mathfrak M_0,\Q).
$$

Moreover, if \(P\) has full support, then $\chiorb(\mathfrak M_0,P)>0.$ In particular, $\chiorb(\mathfrak M_0)>0.$

\end{corollary}

\begin{proof}
Denote \(\mathcal H=\mathcal H_c\cup\mathcal H_\Delta\).
This is a \(P\Gamma\)-invariant locally finite hyperplane arrangement as
shown in the proof of \cite[Theorem~7.1]{ACT} and
\cite[Lemma~7.2 and \S8]{ACT}. Choose a neat finite-index subgroup
\(\Lambda\subset P\Gamma\) and put $B_\Lambda=\Lambda\backslash\mathbb B^{10}.$ The variety \(B_\Lambda\) is smooth and quasi-projective, and
$$
  H_\Lambda=\Lambda\backslash\mathcal H\subset B_\Lambda
$$

is an algebraic divisor, being the pullback of the chordal and
discriminant divisors appearing in \cite[Theorem~7.1]{ACT}. Since
\(B_\Lambda\) is smooth, \(H_\Lambda\) is an effective Cartier divisor.
Thus Corollary~\ref{cor:PEC-arrangement-complement} gives PEC for
$$
  \left[
    \bigl(\mathbb B^{10}\setminus\mathcal H\bigr)/P\Gamma
  \right],
$$

and hence, by the period isomorphism, for \(\mathfrak M_0\).

Now let \(P\) have full support and let \(P_\Lambda\) be its pullback to
\(U_\Lambda=B_\Lambda\setminus H_\Lambda\). Since
\(j\colon U_\Lambda\hookrightarrow B_\Lambda\) is affine,
\(Rj_*P_\Lambda\) is perverse and has the same positive generic rank.
Applying Proposition~\ref{prop:AMP-positivity} to \(Rj_*P_\Lambda\),
together with \(\chi(B_\Lambda)>0\), gives
$$
  \chi(U_\Lambda,P_\Lambda)
  =
  \chi(B_\Lambda,Rj_*P_\Lambda)>0.
$$
Dividing by the degree of the finite \'etale cover
\(U_\Lambda\to\mathfrak M_0^{\mathrm{an}}\) yields
$$
  \chiorb(\mathfrak M_0,P)>0.
$$
Finally, \(\Q_{\mathfrak M_0}[10]\) has full support, so
\(\chiorb(\mathfrak M_0)>0\).
\end{proof}

\subsection{Moduli of smooth cubic fourfolds}

Let
\[
  V_4=\operatorname{Sym}^{3}(\C^{6})^{\vee},
  \qquad
  G_4=\operatorname{PGL}_{6},
\]
and let
\[
  \mathfrak M_4
  =
  \bigl[\mathbb P(V_4)^{\mathrm{sm}}/G_4\bigr]
\]
be the Deligne--Mumford stack of smooth cubic fourfolds.  It is smooth
of dimension \(20\).

Let \(\Lambda_0\) be the primitive middle-cohomology lattice of a cubic
fourfold, of signature \((20,2)\), and let
\[
  \mathcal D
  \subset
  \mathbb P(\Lambda_0\otimes_{\mathbb Z}\C)
\]
be one of the two connected components of the associated type-IV period
domain.  The global monodromy group is
\[
  \Gamma=\operatorname O^*(\Lambda_0).
\]
The arithmetic quotient \(\mathcal D/\Gamma\) is quasi-projective; see
\cite[\S2.1]{LazaZheng}.

There are two \(\Gamma\)-invariant hyperplane arrangements
\(\mathcal H_6\) and \(\mathcal H_2\).  The first is determined by the
short roots, i.e.\ the norm-\(2\) vectors of \(\Lambda_0\), while the
second is determined by the long roots, i.e.\ the norm-\(6\) vectors of
divisibility \(3\).  Their quotients
\[
  \mathcal C_6=\mathcal H_6/\Gamma,
  \qquad
  \mathcal C_2=\mathcal H_2/\Gamma
\]
are Heegner divisors in \(\mathcal D/\Gamma\); see
\cite[Definition~2.1 and Remark~2.2]{LazaZheng}.

By \cite[Theorem~2.3]{LazaZheng}, which cites the global Torelli theorem
of Voisin \cite{VoisinCubicFourfolds}, the work of Hassett on special
cubic fourfolds \cite{HassettSpecialCubicFourfolds}, and the
determination of the period image by Laza \cite[Theorem~1.1]{Laza} and by Looijenga \cite{LooijengaCubicFourfolds}, the period
map induces an isomorphism of quasi-projective varieties
$$
  \mathcal M_4
  \simeq
  \Gamma\backslash
  \bigl(
    \mathcal D\setminus(\mathcal H_2\cup\mathcal H_6)
  \bigr).
$$
As stated immediately after
\cite[Theorem~2.3]{LazaZheng}, identifying the natural orbifold
structures on the two sides is equivalent to the strong global Torelli
theorem. This is supplied by
\cite[Proposition~2.4]{LazaZheng}.  Hence the period map upgrades to an isomorphism of analytic Deligne--Mumford stacks \begin{equation}
\label{eq:cubic-fourfold-stack-period}
  \mathfrak M_4^{\mathrm{an}}
  \simeq
  \left[
    \bigl(
      \mathcal D\setminus(\mathcal H_2\cup\mathcal H_6)
    \bigr)/\Gamma
  \right].
\end{equation}

\begin{corollary}
\label{cor:cubic-fourfold-pec}
The Deligne--Mumford stack \(\mathfrak M_4\) has PEC. Thus
$$
  \chiorb(\mathfrak M_4,P)\geq0
  \qquad
  \text{for every }P\in\Perv(\mathfrak M_4,\Q).
$$

Moreover, if \(P\) has full support, then $\chiorb(\mathfrak M_4,P)>0.$ In particular, $\chiorb(\mathfrak M_4)>0.$

\end{corollary}

\begin{proof}
Set
$$
  \mathcal H_4=\mathcal H_2\cup\mathcal H_6.
$$

By \cite[Definition~2.1]{LazaZheng}, this is a
\(\Gamma\)-invariant arithmetic hyperplane arrangement, and

$$
  \mathcal C_2=\mathcal H_2/\Gamma,
  \qquad
  \mathcal C_6=\mathcal H_6/\Gamma
$$

are algebraic divisors on the quasi-projective variety
\(\mathcal D/\Gamma\); see also
\cite[Remark~2.2]{LazaZheng}.

Choose a neat finite-index subgroup
\(\Lambda\subset\Gamma\), and put $B_\Lambda=\Lambda\backslash\mathcal D.$ By the Baily--Borel theorem, \(B_\Lambda\) is a smooth
quasi-projective variety, and the natural map
$$
  \pi_\Lambda\colon
  B_\Lambda\longrightarrow\mathcal D/\Gamma
$$

is finite and algebraic. Let $H_{4,\Lambda}=\pi_\Lambda^{-1}(\mathcal C_2\cup\mathcal C_6).$ Then \(H_{4,\Lambda}\subset B_\Lambda\) is an algebraic divisor. Since \(B_\Lambda\) is smooth, \(H_{4,\Lambda}\) is an effective Cartier divisor. Hence
$$
  B_\Lambda\setminus H_{4,\Lambda}
  \hookrightarrow B_\Lambda
$$
is an affine open immersion. Thus the hypotheses of
Corollary~\ref{cor:PEC-arrangement-complement} are satisfied, and
$$
  \left[
    \bigl(
      \mathcal D\setminus\mathcal H_4
    \bigr)/\Gamma
  \right]
$$

has PEC.  By the period isomorphism
\eqref{eq:cubic-fourfold-stack-period}, \(\mathfrak M_4\) has PEC.

Moreover, if \(P\in\Perv(\mathfrak M_4,\Q)\) has full support, then its
pullback to $U_\Lambda=B_\Lambda\setminus H_{4,\Lambda}$ has full support and positive generic rank. Since
\(U_\Lambda\hookrightarrow B_\Lambda\) is affine, its direct image is
perverse with the same generic rank. Proposition~\ref{prop:AMP-positivity},
together with \(\chi(B_\Lambda)>0\), therefore gives $\chi(U_\Lambda,P_\Lambda)>0.$ Dividing by the degree of the finite \'etale cover
\(U_\Lambda\to\mathfrak M_4^{\mathrm{an}}\) yields
$$
  \chiorb(\mathfrak M_4,P)>0.
$$

\end{proof}

\section*{Acknowledgments}

The authors thank Donu Arapura for helpful discussions and suggestions.

\bibliographystyle{amsplain}
\bibliography{references}

@article{ACT,
  author  = {Allcock, Daniel and Carlson, James A. and Toledo, Domingo},
  title   = {The Moduli Space of Cubic Threefolds as a Ball Quotient},
  journal = {Mem. Amer. Math. Soc.},
  volume  = {209},
  number  = {985},
  year    = {2011},
  pages   = {xii+70},
  doi     = {10.1090/S0065-9266-10-00591-0},
}

@unpublished{ArapuraMesePatel,
  author = {Arapura, Donu and Mese, Chikako and Patel, Deepam},
  title  = {A Higher Dimensional Log Riemann--Hurwitz Inequality and Rigidity of Covers},
  note   = {Unpublished manuscript},
  year   = {2026},
}

@article{ArapuraPatel,
  author  = {Arapura, Donu and Patel, Deepam},
  title   = {Nonnegativity of Signed Euler Characteristics of Moduli of Curves and Abelian Varieties},
  journal = {Math. Ann.},
  volume  = {394},
  year    = {2026},
  note    = {Article No. 34},
  doi     = {10.1007/s00208-026-03357-0},
}

@book{AMRT,
  author    = {Ash, Armand and Mumford, David and Rapoport, Michael and Tai, Yung-Sheng},
  title     = {Smooth Compactifications of Locally Symmetric Varieties},
  edition   = {2},
  series    = {Cambridge Mathematical Library},
  publisher = {Cambridge University Press},
  address   = {Cambridge},
  year      = {2010},
  note      = {With the collaboration of Peter Scholze},
  doi       = {10.1017/CBO9780511674693},
}

@article{BailyBorel,
  author  = {Baily, Walter L., Jr. and Borel, Armand},
  title   = {Compactification of Arithmetic Quotients of Bounded Symmetric Domains},
  journal = {Ann. of Math. (2)},
  volume  = {84},
  year    = {1966},
  pages   = {442--528},
}

@incollection{BBD82,
  author    = {Beilinson, Alexander A. and Bernstein, Joseph and Deligne, Pierre},
  title     = {Faisceaux pervers},
  booktitle = {Analysis and Topology on Singular Spaces, I (Luminy, 1981)},
  series    = {Ast{\'e}risque},
  volume    = {100},
  publisher = {Soci{\'e}t{\'e} Math{\'e}matique de France},
  address   = {Paris},
  year      = {1982},
  pages     = {5--171},
}

@article{brunebarbe2018symmetric,
  author  = {Brunebarbe, Yohan},
  title   = {Symmetric Differentials and Variations of {H}odge Structures},
  journal = {J. Reine Angew. Math.},
  volume  = {743},
  year    = {2018},
  pages   = {133--161},
  doi     = {10.1515/crelle-2015-0109},
}

@article{Cadorel,
  author  = {Cadorel, Beno{\^i}t},
  title   = {Symmetric Differentials on Complex Hyperbolic Manifolds with Cusps},
  journal = {J. Differential Geom.},
  volume  = {118},
  number  = {3},
  year    = {2021},
  pages   = {373--398},
  doi     = {10.4310/jdg/1625860621},
}

@book{deligne2006equations,
  author    = {Deligne, Pierre},
  title     = {{\'E}quations diff{\'e}rentielles {\`a} points singuliers r{\'e}guliers},
  series    = {Lecture Notes in Mathematics},
  volume    = {163},
  publisher = {Springer-Verlag},
  address   = {Berlin--New York},
  year      = {1970},
  doi       = {10.1007/BFb0061194},
}

@article{Deng,
  author  = {Deng, Ya},
  title   = {A Characterization of Complex Quasi-Projective Manifolds Uniformized by Unit Balls},
  journal = {Math. Ann.},
  volume  = {384},
  year    = {2022},
  pages   = {1833--1881},
  note    = {With an appendix by Ya Deng and Beno{\^i}t Cadorel},
  doi     = {10.1007/s00208-021-02334-z},
}

@book{Dimca,
  author    = {Dimca, Alexandru},
  title     = {Sheaves in Topology},
  series    = {Universitext},
  publisher = {Springer-Verlag},
  address   = {Berlin},
  year      = {2004},
}

@article{fujino2019semipositivity,
  author  = {Fujino, Osamu and Fujisawa, Taro},
  title   = {On Semipositivity Theorems},
  journal = {Math. Res. Lett.},
  volume  = {26},
  number  = {5},
  year    = {2019},
  pages   = {1359--1382},
}

@article{harris1989functorial,
  author  = {Harris, Michael},
  title   = {Functorial Properties of Toroidal Compactifications of Locally Symmetric Varieties},
  journal = {Proc. London Math. Soc. (3)},
  volume  = {59},
  number  = {1},
  year    = {1989},
  pages   = {1--22},
  doi     = {10.1112/plms/s3-59.1.1},
}

@article{HassettSpecialCubicFourfolds,
  author  = {Hassett, Brendan},
  title   = {Special Cubic Fourfolds},
  journal = {Compos. Math.},
  volume  = {120},
  number  = {1},
  year    = {2000},
  pages   = {1--23},
  doi     = {10.1023/A:1001706324425},
}

@book{Helgason,
  author    = {Helgason, Sigurdur},
  title     = {Differential Geometry, Lie Groups, and Symmetric Spaces},
  series    = {Graduate Studies in Mathematics},
  volume    = {34},
  publisher = {American Mathematical Society},
  address   = {Providence, RI},
  year      = {2001},
  note      = {Corrected reprint of the 1978 original},
  doi       = {10.1090/gsm/034},
}

@article{hwang2004volumes,
  author  = {Hwang, Jun-Muk},
  title   = {On the Volumes of Complex Hyperbolic Manifolds with Cusps},
  journal = {Internat. J. Math.},
  volume  = {15},
  number  = {6},
  year    = {2004},
  pages   = {567--572},
  doi     = {10.1142/S0129167X04002442},
}

@book{KashiwaraSchapira,
  author    = {Kashiwara, Masaki and Schapira, Pierre},
  title     = {Sheaves on Manifolds},
  series    = {Grundlehren der mathematischen Wissenschaften},
  volume    = {292},
  publisher = {Springer-Verlag},
  address   = {Berlin},
  year      = {1990},
}

@article{LaszloOlsson,
  author  = {Laszlo, Yves and Olsson, Martin},
  title   = {Perverse {$t$}-Structure on {A}rtin Stacks},
  journal = {Math. Z.},
  volume  = {261},
  number  = {4},
  year    = {2009},
  pages   = {737--748},
  doi     = {10.1007/s00209-008-0348-z},
}

@article{Laza,
  author  = {Laza, Radu},
  title   = {The Moduli Space of Cubic Fourfolds via the Period Map},
  journal = {Ann. of Math. (2)},
  volume  = {172},
  number  = {1},
  year    = {2010},
  pages   = {673--711},
  doi     = {10.4007/annals.2010.172.673},
}

@article{LazaZheng,
  author  = {Laza, Radu and Zheng, Zhiwei},
  title   = {Automorphisms and Periods of Cubic Fourfolds},
  journal = {Math. Z.},
  volume  = {300},
  number  = {2},
  year    = {2022},
  pages   = {1455--1507},
  doi     = {10.1007/s00209-021-02810-x},
}

@book{LazarsfeldPositivityII,
  author    = {Lazarsfeld, Robert},
  title     = {Positivity in Algebraic Geometry II: Positivity for Vector Bundles, and Multiplier Ideals},
  series    = {Ergebnisse der Mathematik und ihrer Grenzgebiete. 3. Folge},
  volume    = {49},
  publisher = {Springer-Verlag},
  address   = {Berlin},
  year      = {2004},
  doi       = {10.1007/978-3-642-18810-7},
}

@article{LooijengaCubicFourfolds,
  author  = {Looijenga, Eduard},
  title   = {The Period Map for Cubic Fourfolds},
  journal = {Invent. Math.},
  volume  = {177},
  number  = {1},
  year    = {2009},
  pages   = {213--233},
  doi     = {10.1007/s00222-009-0178-6},
}

@article{MargulisArithmeticity,
  author  = {Margulis, G. A.},
  title   = {Arithmeticity of the Irreducible Lattices in the Semisimple Groups of Rank Greater than One},
  journal = {Invent. Math.},
  volume  = {76},
  number  = {1},
  year    = {1984},
  pages   = {93--120},
  doi     = {10.1007/BF01388494},
}

@incollection{MilneShimura,
  author    = {Milne, J. S.},
  title     = {Introduction to {S}himura Varieties},
  booktitle = {Harmonic Analysis, the Trace Formula, and Shimura Varieties},
  series    = {Clay Mathematics Proceedings},
  volume    = {4},
  publisher = {American Mathematical Society},
  address   = {Providence, RI},
  year      = {2005},
  pages     = {265--378},
  url       = {https://www.jmilne.org/math/xnotes/svi2004.pdf},
}

@incollection{Mok,
  author    = {Mok, Ngaiming},
  title     = {Projective Algebraicity of Minimal Compactifications of Complex-Hyperbolic Space Forms of Finite Volume},
  booktitle = {Perspectives in Analysis, Geometry, and Topology},
  series    = {Progress in Mathematics},
  volume    = {296},
  publisher = {Birkh{\"a}user/Springer},
  address   = {New York},
  year      = {2012},
  pages     = {331--354},
}

@article{mumford1977hirzebruch,
  author  = {Mumford, David},
  title   = {Hirzebruch's Proportionality Theorem in the Non-Compact Case},
  journal = {Invent. Math.},
  volume  = {42},
  year    = {1977},
  pages   = {239--272},
  doi     = {10.1007/BF01389790},
}

@article{norimatsu1978kodaira,
  author  = {Norimatsu, Yoshiki},
  title   = {Kodaira Vanishing Theorem and Chern Classes for {$\partial$}-Manifolds},
  journal = {Proc. Japan Acad. Ser. A Math. Sci.},
  volume  = {54},
  number  = {4},
  year    = {1978},
  pages   = {107--108},
}

@article{VoisinCubicFourfolds,
  author  = {Voisin, Claire},
  title   = {Th{\'e}or{\`e}me de {T}orelli pour les cubiques de {$\mathbf{P}^5$}},
  journal = {Invent. Math.},
  volume  = {86},
  number  = {3},
  year    = {1986},
  pages   = {577--601},
  doi     = {10.1007/BF01389270},
}

@book{WitteMorris,
  author    = {Witte Morris, Dave},
  title     = {Introduction to Arithmetic Groups},
  publisher = {Deductive Press},
  year      = {2015},
  url       = {https://deductivepress.ca/dmorris/books/IntroArithGrps/IntroArithGrps-FINAL.pdf},
}

@article{WuZhou,
  author  = {Wu, Lei and Zhou, Peng},
  title   = {Log {$\mathscr D$}-Modules and Index Theorems},
  journal = {Forum Math. Sigma},
  volume  = {9},
  year    = {2021},
  pages   = {e3},
  note    = {32 pp.},
  doi     = {10.1017/fms.2020.62},
}

\end{document}